\documentclass{cimart}

\usepackage{longtable}
\allowdisplaybreaks

\title[A new approach to defining cochain complexes for Dendriform and pre-Lie algebras]{A new approach to defining cochain complexes for \\ dendriform and pre-Lie algebras}

\authors{Hassan Alhussein}

\authorinfo{Siberian State University of Telecommunication and Informatics Science, Novosibirsk State University of Economics and Management, and Novosibirsk State University, Russia}{k.alkhussein@g.nsu.ru}

\abstract{%
    Our constructions provide a systematic way to study cohomology of pre-algebraic structures via classical cohomology, simplifying computations and enabling the use of established techniques.
    }



\VOLUME{34}
\YEAR{2026}
\ISSUE{1}
\NUMBER{10}
\DOI{https://doi.org/10.46298/cm.17352}
\licence{CC BY 4.0}
\editinfo{January 22, 2026}{March 8, 2026}{Pasha Zusmanovich}
\acknowledgments{The author is grateful to  Kolesnikov~P.~S. and Gubarev~V.~Yu. for discussions and useful comments.}

\begin{document}

\section{Introduction}

Algebraic structures and their cohomology theories play a fundamental role in understanding deformations, extensions, and classification problems in mathematics. Among these structures, Perm algebras, pre-associative algebras, and pre-Lie algebras have garnered significant interest due to their connections with associative algebras, Lie algebras, and operad theory. This paper explores the cohomology theories of these algebras and establishes relationships between them via tensor product constructions.

A {Perm algebra} is an algebra satisfying the permutation identity:  
\[
(a \cdot_A b) \cdot_A c = a \cdot_A (b \cdot_A c) = a \cdot_A (c \cdot_A b).
\]
The cohomology of Perm algebras, introduced in \cite{Gnedbaye2020}, provides a framework for studying deformations and extensions. A key result is the embedding of Perm cochain complexes into Hochschild cochain complexes, inducing a map in cohomology that is an isomorphism in low degrees but not necessarily surjective in higher degrees.

A {pre-associative algebra} (or {dendriform algebra}) is equipped with two operations $\prec$ and $\succ$ that split the associativity of the total product $a \cdot_B b = a \prec b + a \succ b$. Such algebras arise in the study of Rota-Baxter algebras, shuffle algebras, and combinatorial Hopf algebras \cite{Loday}. Their cohomology, defined via a complex of multilinear maps satisfying compatibility conditions, encodes deformation theory and operadic structures \cite{Das}.

A {pre-Lie algebra} (or {right-symmetric algebra}) satisfies the identity:  
\[
(a \cdot_P b) \cdot_P c - a \cdot_P (b \cdot_P c) = (a \cdot_P c) \cdot_P b - a \cdot_P (c \cdot_P b).
\]
These algebras appear in deformation quantization \cite{Gerstenhaber}. The cohomology of pre-Lie algebras, introduced in \cite{Dzhumadil’daev}, involves a differential that accounts for both the pre-Lie product and the induced Lie bracket.

    In \cite{kolesnikov}, it is shown that for a given free Perm algebra \(A\) and algebra \(B\) with two bilinear operations \(\prec, \succ\), 
the tensor product \(A \otimes B\) is associative if and only if \((B, \prec, \succ)\) is a pre-associative algebra.
      Similarly in \cite{GubarevKolesnikov}, for a free Perm algebra \(A\) and  algebra \(P\) with a bilinear operation \(\cdot_P\), 
the tensor product \(A \otimes P\) is a Lie algebra if and only if \((P, \cdot_P)\) is a pre-Lie algebra.

 In this paper, we construct explicit injective cochain maps that embed the cohomology of pre-associative and pre-Lie algebras into the cohomology of tensor products with free Perm algebras, revealing a new structural connection between these theories:
\begin{itemize}
    \item When $A$ is a free Perm algebra: from pre-associative cohomology $C^*_{\text{dend}}(B,N)$ to Hochschild cohomology $C^*_{\text{Hoch}}(A \otimes B, A \otimes N)$.
    \item From pre-Lie cohomology $C^*_{\text{Pre}}(P,N)$ to Lie algebra cohomology $C^*_{\text{Lie}}(A \otimes P, A \otimes N)$.
\end{itemize}

These embeddings induce long exact sequences relating the respective cohomologies, providing a tool to compare deformation theories.

This injective cochain maps reduce the computation of dendriform (pre-associative) and pre-Lie cohomologies to classical Hochschild and Lie cohomologies.

\section{Right Perm algebras and pre-associative  algebras}

\begin{definition}[\!\!\cite{Chapoton2001}]
An associative algebra $(A, \cdot_A)$ is called a right \textbf{Perm algebra} if it satisfies the permutation identity:
\[
(a \cdot_A b) \cdot_A c = a \cdot_A (b \cdot_A c)=  a \cdot_A (c \cdot_A b),\quad  a,b,c \in A .
\]
A representation of $A$ is an $A$-bimodule $M$ 
satisfying 
\[
    (mb)c=m(bc)=m(cb), \quad
    (am)b=a(mb)=(ab)m
    \]
for all $m\in M$, $a,b\in A$.
\end{definition}

Let $(A, \cdot)$ be a Perm algebra
with a representation on an $A$-bimodule $M$.
The \textbf{Perm cochain complex} 
$(C^*_{\text{perm}}(A,M), \delta_{\text{perm}})$
is given by:
\[
C_{\text{perm}}^n=C^n_{\text{perm}}(A,M) :=\{f\in Hom_{\Bbbk}(A^{\otimes n}, M)\quad \text{satisfying (1), (2), (3), (4)\ and\ (5)}\}
\]
\begin{align}
& af(a_{\gamma(1)},\dots,a_{\gamma(n)}) = af(a_{1},\dots,a_{n});\quad \gamma \in S_n,
\\
& f(a_1,\dots,ba_i,\dots,a_n)= f(a_1,\dots,a_ib,\dots,a_n);\quad i=2,\dots,n,
\\
& af(a_1b,a_2,\dots,a_n)=af(ba_1,a_2,\dots,a_n),\\
& abf(a_1,\dots,a_n)=aa_jf(a_1,\dots,a_{j-1},b,a_{j+1},\dots,a_n),
\\
& af(a_1,\dots,a_{j-1}b,\dots,a_n)=af(a_1b,a_2,\dots,a_n).
\end{align}
The differential 
$\delta_{\text{perm}}: C_{\text{perm}}^n \to C_{\text{perm}}^{n+1}$ is defined by the same rule as the classical Hochschild differential $\delta_{Hoch}$:
\begin{align*}
(\delta_{\text{perm}} f)(a_1, \dots, a_{n+1}) &
= a_1 f(a_2, \dots, a_{n+1}) \\
&+ \sum_{i=1}^n (-1)^i f(a_1, \dots, a_i \cdot a_{i+1}, \dots, a_{n+1}) \\
&+ (-1)^{n+1} f(a_1, \dots, a_n) a_{n+1}
\end{align*}
The \textbf{Perm cohomology} is $H^*_{\text{perm}}(A,M) = H(C^*_{\text{perm}}(A,M), \delta_{\text{perm}})$ .

1-cohomology corresponds to outer derivations, 2-cohomology describes abelian extensions; for details see \cite{Gnedbaye2020}.
\begin{definition}[see \cite{Loday}]\label{defn:Pre-As}
An algebra $(B, \prec, \succ)$ with two binary operations is called a \textbf{pre-associative  algebra} (or dendriform algebra; also known in the literature as a non-commutative half-shuffle algebra) if the following identities hold for all $x,y,z \in B$:
\begin{align*}
(x \prec y) \prec z &= x \prec (y \prec z + y \succ z), \\
(x \succ y) \prec z &= x \succ (y \prec z), \\
x \succ (y \succ z) &= (x \prec y + x \succ y) \succ z.
\end{align*}
In particular, the \textbf{total product} $x \cdot_B y := x \prec y + x \succ y$ is associative.
\end{definition}

\begin{example}\label{exp:dend}
Let    $B$ be an $n$-dimensional pre-associative algebra with basis $\{e_1,\dots,e_n\}$ and operations:
\begin{align*}
e_i \prec e_1 &= e_i \ \ \textit{for all $i=2,\dots,n$,}\\
e_1 \succ e_1 &= e_1,
\end{align*}
All other products are zero.
\end{example}

\begin{definition}[see \cite{Loday}]
A representation of a pre-associative algebra $B$ on a vector 
space $N$ is defined as two left actions and two right actions satisfying nine obvious identities similar to those in Definition \ref{defn:Pre-As}.
\end{definition}
The notion of a cohomology for pre-associative algebras
was proposed by \cite{Das}. 
In this section, we recall the original definition which is quite technical. In Section~3, we show that when \(A\) is a free Perm algebra, there exists an injective cochain map
\[
\Psi \colon C^*_{\mathrm{dend}}(B,N) \longrightarrow C^*_{\mathrm{Hoch}}(A \otimes B, A \otimes N)
\]
from the pre-associative complex of \(B\) into the Hochschild complex of \(A \otimes B\). 
This embedding not only simplifies computations but also establishes a universal connection between pre-associative and Hochschild cohomologies, which we regard as a simpler and more conceptual approach.

Denote by $C_n = \{1, 2,\dots, n\}$ the set of the first $n$ natural numbers.
For any $m, n \geq 1$ and $1 \leq i \leq m$, define the maps $R_0,R_1,\dots,R_m$ as follows:
\[
R_0(m,i,n): C_{m+n-1} \to C_m,
\]
\[
R_0(m,i,n)(r)=
\begin{cases}
r, & r\leq i-1,\\
i, & i\leq r \leq i+n-1,\\
r-n+1, & i+n\leq r\leq m+n-1;
\end{cases}
\]
\[
R_i(m,i,n): C_{m+n-1} \to \Bbbk[C_m],
\]
\[
R_i(m,i,n)(r)=
\begin{cases}
1+\dots+n, & r\leq i-1,\\
r-i+1, & i\leq r \leq i+n-1,\\
1+\dots+n, & i+n\leq r\leq m+n-1.
\end{cases}
\]

For a pre-associative algebra $(B, \prec, \succ )$, define  
$\pi_B: \Bbbk [C_2]\otimes B\otimes B\to B$ 
by
\[
\pi_B(r,a,b)=
\begin{cases}
    a\prec b, & r=1,\\
    a\succ b, &  r=2.
\end{cases}
\]
To describe a representation of $B$ on $N$ in a more compact form, we may define two maps 
\[
\theta_1: \Bbbk[C_2]\otimes (B\otimes N)\to N,
\quad 
\theta_2:\Bbbk[C_2]\otimes (N\otimes B)\to N
\]
as follows:
\[
\theta_1(r,a,n)=
\begin{cases}
    a\prec  n, & r=1,\\
    a \succ n, &  r=2;
\end{cases}
\quad
\theta_2(r,n,a)=
\begin{cases}
    n\prec a, & r=1,\\
    n\succ a, &  r=2.
\end{cases}
\]

\begin{definition}[\!\!\cite{Das}]\label{defn:Pre-As-Cochain}
Let $(B, \prec, \succ)$ be a pre-associative algebra and 
$N$ be a representation of $B$. 
We define the space $C^n_{dend}(B,N)$
of $n$-cochains of $B$ with coefficients in $N$ by
\[
C^n_{dend}(B,N)=Hom_\Bbbk(\Bbbk [C_n]\otimes B^{\otimes_n},N).
\]
The differential 
\[
\delta_{{dend}}: C^n_{dend}(B,N) \to C^{n+1}_{dend}(B,N) 
\]
is given by
\begin{multline}
(\delta_{{dend}}f)(r,a_1,\dots,a_{n+1})=\theta_1(R_0(2,1,n)(r),a_1,f(R_2(2,1,n)(r),a_2,\dots,a_{n+1})) \\
+\sum_{i=1}^n(-1)^i
 f(R_0(n,i,2)(r),a_1,\dots,a_{i-1},\pi_A(R_i(n,i,2)(r),a_i,a_{i+1}), a_{i+2},\dots,a_{n+1})  \\
+ (-1)^{n+1}\theta_2(R_0(2,n,1)(r),f(R_1(2,n,1)(r),a_1,\dots,a_{n}),a_{n+1}),
\end{multline}
for $r\in C_{n+1}$ and $a_1,\dots,a_{n+1}\in B$. 
The corresponding cohomology groups are denoted by 
$H_{dend}^n (B, N)$, for $n \geq 1 $.
\end{definition}

 The  pre-associative differential 
may be expressed via
\begin{longtable}{@{}l@{}l@{}}

$\displaystyle (\delta_{dend}f)$
& $(1,b_1,\dots,b_{n+1}) =\displaystyle \sum_{i=1}^n a_1\prec f(i,b_2,\dots,b_{n+1}) - f(1,b_1\prec b_2,\dots,b_{n+1})$ \\[4pt]
& $\displaystyle + \sum_{i=2}^n(-1)^i f(1,b_1,\dots, b_i*b_{i+1},\dots,b_{n+1}) + (-1)^{n+1} f(1,b_1,b_2,\dots,b_n)\prec b_{n+1}$ \\[4pt]

& \\[8pt]

$\displaystyle (\delta_{dend}f)$
& $(2,b_1,\dots,b_{n+1}) =\displaystyle b_1\succ f(1,b_2,\dots,b_{n+1}) - f(1,b_1\succ b_2,b_3,\dots,b_{n+1})$ \\[4pt]
& $\displaystyle + f(2,b_1,b_2\prec b_3,\dots,b_{n+1}) + \sum_{i=3}^{n}(-1)^i f(2,b_1,\dots, b_i*b_{i+1},\dots,b_{n+1})$ \\[4pt]
& $\displaystyle + (-1)^{n+1} f(2,b_1,\dots,b_n)\prec b_{n+1}$
\end{longtable}

   \[
   \begin{aligned}
  (\delta_{dend}f)(r,b_1, \dots,&  b_{n+1}) =   \sum_{i=1}^{r-1}b_1\succ f(r-1,b_2,\dots,b_{n+1})\\
   &+\sum_{i=1}^{r-2}(-1)^if(r-1,b_1,\dots, b_i*b_{i+1},\dots,b_{n+1})\\
    &+(-1)^{r-1} f(r-1,b_1,\dots,b_{r-1}\succ b_r,\dots,b_{n+1})\\
     &+(-1)^{r} f(r,b_1,\dots,b_{r}\prec b_{r+1},\dots,b_{n+1})\\
     &+ \sum_{i=r+1}^{n}(-1)^if(r,b_1,\dots, b_i*b_{i+1},\dots,b_{n+1})\\
     &+\sum_{i=r+1}^{n}(-1)^{n+1}f(r,b_1,\dots,b_{n})\prec b_{n+1}\\
     (\delta_{dend}f)(n+1,b_1, \dots,&  b_{n+1}) = b_1\succ f(n,b_2,\dots,b_{n+1})\\
     &+\sum_{i=1}^{n-1}(-1)^i f(n,b_1,\dots, b_i*b_{i+1},\dots,b_{n+1})\\
     &+(-1)^{n} f(n,b_1,b_2,\dots,b_n\succ b_{n+1})\\
     &+\sum_{i=1}^n(-1)^{n+1}f(i,b_1,b_2,\dots,b_n)\succ b_{n+1}.
\end{aligned}
\]
for all $r=3,\dots,n$.

In particular, for $f\in C^1_{dend}(B,N)$ we have
 \begin{align*}
& (\delta_{{dend}}f)(1;a,b)=a\prec f(1,b)
-f(1,a\prec b)+(1,a)\prec b,\\
& (\delta_{{dend}}f)(2,a,b)=a\succ f(1,b)
-f(1,a\succ b)+f(1,a)\succ b.
\end{align*} 

Similarly, if $f\in C^2_{dend}(B,N)$ then
\begin{equation}
\begin{aligned}
 (\delta_{{dend}}f)(1;a,b,c) &{}= a\prec f(1,b,c)+a\prec f(2,b,c)-f(1,a\prec b,c) \\
 &{} + f(1,a,b\prec c)+f(1,a,b\succ c)-f(1,a,b)\prec c,\\
(\delta_{{dend}}f)(2;a,b,c) &{} =a\succ f(1,b,c)-f(1,a\succ b,c)+f(2,a,b\prec c) - f(2,a,b)\prec c,\\
(\delta_{{dend}}f)(3;a,b,c) &{}=a\succ f(2,b,c)-f(2,a\succ b,c)-f(2,a\prec b,c)\\ 
&{}+ f(2,a,b\succ c)-f(1,a,b)\succ c-f(2,a,b)\succ c.
\end{aligned}
\end{equation}

\begin{example}\label{2.7}
Let us compute \(H^2(B,B)\), where \(B\) is the pre-associative algebra defined in \ref{exp:dend}.

A 2-cochain \(f \in C^2(B,B)\) is defined by bilinear functions:
\[
f(1, e_i, e_j) = \sum_{k=1}^n \alpha_{ij}^k e_k, \quad f(2, e_i, e_j) = \sum_{k=1}^n \beta_{ij}^k e_k.
\]

The cocycle condition \(\delta_{\text{dend}} f = 0\) is given by the following three equations for all \(a,b,c \in B\):
\[
\begin{aligned}
(\delta_{dend} f)(1;a,b,c) &= a \prec f(1,b,c) + a \prec f(2,b,c) - f(1,a \prec b,c) \\
&\quad + f(1,a,b \prec c) + f(1,a,b \succ c) - f(1,a,b) \prec c, \\
(\delta_{dend} f)(2;a,b,c) &= a \succ f(1,b,c) - f(1,a \succ b,c) \\
&\quad + f(2,a,b \prec c) - f(2,a,b) \prec c, \\
(\delta_{dend} f)(3;a,b,c) &= a \succ f(2,b,c) - f(2,a \succ b,c) - f(2,a \prec b,c) \\
&\quad + f(2,a,b \succ c) - f(1,a,b) \succ c - f(2,a,b) \succ c.
\end{aligned}
\]

Evaluating these conditions on all basis triples yields the general solution for a cocycle:
\[
\begin{aligned}
\alpha^1_{11}&=0, & \alpha^k_{1j}&=0,\ \forall j\geq 2,\ k\geq 1\\
\alpha^k_{i1}&=0,\ \forall i\geq 2, k\geq 1 & \alpha^k_{ij}&=0,\ \forall i,j,k\geq 2\\
\beta^1_{11}&=0, & \beta^k_{1j}&=0,\ \forall j\geq 2,\ k\geq 1\\
\beta^k_{i1}&=0,\ \forall i\geq 2, k\geq 1 & \beta^k_{ij}&=0,\ \forall i,j,k\geq 2\\
\alpha^k_{11}&+\beta^k_{11}=0,\ \forall k\geq 2.
\end{aligned}
\]
Thus, the space of 2-cocycles \(Z^2(B,B)\) is spanned by maps of the form:
\[
f(1, e_1, e_1) = \sum_{k=2}^n\alpha_{11}^k e_k, \quad f(2, e_1, e_1) = -\sum_{k=2}^n\alpha_{11}^k e_k.
\]

Now, let \(g \in C^1(B,B)\) be a 1-cochain defined by \(g(1, e_i) = \sum_{k=1}^n c_i^k e_k\). Its coboundary \(\delta_{\text{dend}} g\) is:
\[
\begin{aligned}
(\delta_{\text{dend}} g)(1; e_1, e_1) &= \sum_{k=2}^n c_1^k e_k, \\
(\delta_{\text{dend}} g)(2; e_1, e_1) &= -\sum_{k=1}^n c_1^k e_k.
\end{aligned}
\]

Comparing with the general cocycle, we find that every cocycle \(f\) is also a coboundary \(\delta_{\text{dend}} g\), achieved by setting \(c_1^1 = 0\) and \(c_1^k = \alpha_{11}^k\) for \(k \geq 2\). Therefore, \(H^2(B,B) = 0\).
\end{example}
\section{Associative and pre-associative   algebra}
\begin{theorem}[see {\cite[Example 2]{kolesnikov}}]\label{4.1}
Let $(A, \cdot_A)$ be a  Perm algebra and $B$ be a vector space equipped with two bilinear operations $\prec, \succ : B \otimes B \to B$. 
Define the tensor product algebra $A \otimes B$ with product
\[
(a_1 \otimes b_1) \cdot (a_2 \otimes b_2) := (a_1 \cdot_A a_2) \otimes (b_1 \prec b_2) + (a_2 \cdot_A a_1) \otimes (b_1 \succ b_2).
\]
Then:
\begin{enumerate}
    \item If $(B, \prec, \succ)$ is a pre-associative algebra, then $A \otimes B$ is associative.
    \item Conversely, if $A \otimes B$ is associative and $A$ is \textbf{free as a Perm algebra}, then $(B, \prec, \succ)$ is a pre-associative algebra.
\end{enumerate}
In particular, when $A$ is a free Perm algebra, $A \otimes B$ is associative \textbf{if and only if} $(B, \prec, \succ)$ is a pre-associative algebra.
\end{theorem}
\begin{theorem}
Let $A$ be a free Perm algebra, let $(B,\prec,\succ)$ be a pre-associative  algebra, and let $N$ be a vector space. Then the tensor product $A \otimes N$ with action
\[
(a \otimes n) \cdot (b \otimes c) := (a  b) \otimes (n \prec c) + (b  a)\otimes(n \succ c)
\]
\[
(b \otimes c) \cdot (a \otimes n) := (b  a) \otimes (c \prec n) + (a  b)\otimes(c \succ n)
\]
is a module over $A\otimes B$ if and only if $N$ is a module over $B$. 
\end{theorem}
\begin{proof}
    This follows using an argument similar to the proof of Theorem  \ref{4.1}.
\end{proof}
\begin{remark}
    The cohomology of an algebra 
$A\otimes B$
 with values in a bimodule 
$A\otimes N$ is a Hochschild cohomology.
\end{remark}
\begin{theorem}
Let \((A, \cdot_A)\) be a \textbf{Perm algebra}, let \((B, \prec, \succ)\) be a \textbf{pre-associative algebra}, and let \(N\) be a \(B\)-bimodule. Then there exists a cochain map
\[
\Psi : C^*_{\mathrm{dend}}(B, N) \rightarrow C^*_{\mathrm{Hoch}}(A \otimes B, A \otimes N)
\]
from the pre-associative cochain complex of \(B\) with coefficients in \(N\) to the Hochschild cochain complex of the tensor product algebra \(A \otimes B\) with coefficients in \(A \otimes N\).
\end{theorem}
\begin{proof}
  Given a pre-associative  cochain $f \in C^n_{\mathrm{dend}}(B, N)$, for each $x_1,\dots,x_n\in A$ and each $b_1,\dots,b_n\in B$, define: 
\begin{longtable}{@{}l@{}l@{}}
$\displaystyle \Psi(f)\big((x_1 \otimes b_1, x_2 \otimes b_2,\dots,$
& $x_{n}\otimes b_{n})\big) :=\displaystyle x_1x_2\dots x_{n} \otimes f(1,b_1 , b_2,\dots,b_{n})$ \\[4pt]
& $\displaystyle + x_2x_1x_3\dots x_{n} \otimes f(2,b_1, b_2,\dots,b_{n})$ \\[4pt]
& $\displaystyle + \sum_{i=2}^{n-2}x_{i+1}x_i\dots x_{1}x_{i+2}x_{i+3}\dots x_{n}\otimes f(i+1,b_1,b_2,\dots,b_n)$ \\[4pt]
& $\displaystyle + x_nx_{n-1}\dots x_2x_1\otimes f(n,b_1,b_2,\dots,b_{n})$
\end{longtable}

We will show that $\Psi$ commutes with the  differentials:
\[
\delta_{HH}(\Psi f)=\Psi(\delta_{dend}f)
\]
For $n=3$ and $x =x_1  \otimes b_1$, $y = x_2 \otimes b_2$, $z = x_3 \otimes b_3$, we have:

\[
\delta_{HH}(\Psi f)(x,y,z) = \underbrace{x \ast (\Psi f)(y,z)}_{(1)} - \underbrace{(\Psi f)(x \ast y, z)}_{(2)} + \underbrace{(\Psi f)(x, y \ast z)}_{(3)} - \underbrace{(\Psi f)(x,y) \ast z}_{(4)}
\]

\textbf{Term (1): $x \ast (\Psi f)(y,z)$}

\begin{align*}
&= (x_1 \otimes b_1) \ast \left[x_2 \cdot_A x_3 \otimes f(1,b_2,b_3) + x_3 \cdot_A x_2 \otimes f(2,b_2,b_3)\right] \\
&= (x_1 \cdot_A (x_2 \cdot_A x_3)) \otimes (b_1 \prec f(1,b_2,b_3)) + ((x_2 \cdot_A x_3) \cdot_A x_1) \otimes (b_1 \succ f(1,b_2,b_3)) \\
&\quad + (x_1 \cdot_A (x_3 \cdot_A x_2)) \otimes (b_1 \prec f(2,b_2,b_3)) + ((x_3 \cdot_A x_2) \cdot_A x_1) \otimes (b_1 \succ f(2,b_2,b_3))
\end{align*}

\textbf{Term (2): $-(\Psi f)(x \ast y, z)$}

\begin{align*}
&= -(\Psi f)\left((x_1 \cdot_A x_2) \otimes (b_1 \prec b_2) + (x_2 \cdot_A x_1) \otimes (b_1 \succ b_2), x_3 \otimes b_3\right) \\
&= - (x_1 \cdot_A x_2) \cdot_A x_3 \otimes f(1, b_1 \prec b_2, b_3) - x_3 \cdot_A (x_1 \cdot_A x_2) \otimes f(2, b_1 \prec b_2, b_3) \\
&\quad - (x_2 \cdot_A x_1) \cdot_A x_3 \otimes f(1, b_1 \succ b_2, b_3) - x_3 \cdot_A (x_2 \cdot_A x_1) \otimes f(2, b_1 \succ b_2, b_3)
\end{align*}

\textbf{Term (3): $+(\Psi f)(x, y \ast z)$}

\begin{align*}
&= (\Psi f)\left(x_1 \otimes b_1, (x_2 \cdot_A x_3) \otimes (b_2 \prec b_3) + (x_3 \cdot_A x_2) \otimes (b_2 \succ b_3)\right) \\
&= x_1 \cdot_A (x_2 \cdot_A x_3) \otimes f(1, b_1, b_2 \prec b_3) + (x_2 \cdot_A x_3) \cdot_A x_1 \otimes f(2, b_1, b_2 \prec b_3) \\
&\quad + x_1 \cdot_A (x_3 \cdot_A x_2) \otimes f(1, b_1, b_2 \succ b_3) + (x_3 \cdot_A x_2) \cdot_A x_1 \otimes f(2, b_1, b_2 \succ b_3)
\end{align*}

\textbf{Term (4): $-(\Psi f)(x,y) \ast z$}

\begin{align*}
&= -\left[x_1\cdot_A x_2 \otimes f(1,b_1,b_2) + x_2 \cdot_A x_1 \otimes f(2,b_1,b_2)\right] \ast (x_3 \otimes b_3) \\
&= - (x_1 \cdot_A x_2) \cdot_A x_3 \otimes (f(1,b_1,b_2) \prec b_3) - x_3 \cdot_A (x_1 \cdot_A x_2) \otimes (f(1,b_1,b_2) \succ b_3) \\
&\quad - (x_2 \cdot_A x_1) \cdot_A x_3 \otimes (f(2,b_1,b_2) \prec b_3) - x_3 \cdot_A (x_2 \cdot_A x_1) \otimes (f(2,b_1,b_2) \succ b_3)
\end{align*}
Now, let's collect all terms from (1)-(4) and group them by which $(\delta_{dend} f)(i)$ they contribute to.

{Terms Contributing to $(\delta_B f)(1, b_1, b_2, b_3)$:}

\begin{itemize}
\item From (1): $b_1 \prec f(1,b_2,b_3)$ and $b_1 \prec f(2,b_2,b_3)$
\item From (2): $-f(1, b_1 \prec b_2, b_3)$
\item From (3): $+f(1, b_1, b_2 \succ b_3)$ and $+f(1, b_1, b_2 \prec b_3)$
\item From (4): $-f(1,b_1,b_2) \prec b_3$
\end{itemize}

Combined:
\begin{align*}
(\delta_{dend} f)(1, b_1, b_2, b_3) &= b_1 \prec f(1,b_2,b_3)+b_1 \prec f(2,b_2,b_3) - f(1, b_1 \prec b_2, b_3) \\
&\quad + f(1, b_1, b_2 \succ b_3) + f(1, b_1, b_2 \prec b_3) - f(1,b_1,b_2) \prec b_3
\end{align*}

\textbf{Terms Contributing to $(\delta_{dend} f)(2, b_1, b_2, b_3)$:}

\begin{itemize}
\item From (1): $b_1 \succ f(1,b_2,b_3)$
\item From (2): $-f(1, b_1 \succ b_2, b_3)$
\item From (3): $+f(2, b_1, b_2 \prec b_3)$
\item From (4): $-f(2,b_1,b_2) \prec b_3$
\end{itemize}

Combined:
\begin{align*}
(\delta_{dend} f)(2, b_1, b_2, b_3) &= b_1 \succ f(1,b_2,b_3) - f(1, b_1 \succ b_2, b_3) \\
&\quad + f(2, b_1, b_2 \prec b_3) - f(2,b_1,b_2) \prec b_3
\end{align*}

\textbf{Terms Contributing to $(\delta_{dend} f)(3, b_1, b_2, b_3)$:}

\begin{itemize}
\item From (1): $b_1 \succ f(2,b_2,b_3)$ 
\item From (2):  $-f(1, b_1 \succ b_2, b_3)$ and $-f(2, b_1 \prec b_2, b_3)$
\item From (3): $+f(2, b_1, b_2 \succ b_3)$
\item From (4): $-f(1,b_1,b_2) \succ b_3$ and $-f(2,b_1,b_2) \succ b_3$
\end{itemize}

Combined:
\begin{align*}
(\delta_{dend} f)(3, b_1, b_2, b_3) &= b_1 \succ f(2,b_2,b_3)  - f(2, b_1 \succ b_2, b_3) \\
&\quad - f(2, b_1 \prec b_2, b_3) + f(2, b_1, b_2 \succ b_3) \\
&\quad - f(1,b_1,b_2) \succ b_3 - f(2,b_1,b_2) \succ b_3
\end{align*}
By using the Perm identity, which allows reordering of monomials in $A$
 modulo the relation $(x_ix_j)x_k=x_i(x_jx_k)=x_i(x_kx_j)$, we see that:
\begin{align*}
(\delta_{HH}(\Psi(f))\big(x_1 \otimes b_1, x_2 \otimes b_2,x_{3}\otimes b_{3}\big) := x_1x_2 x_{3} \otimes (\delta_{dend}f)(1,b_1 , b_2,b_{3}) \\+ x_2x_1x_3 \otimes (\delta_{dend}f)(2,b_1, b_2,b_{3})
+x_3 x_2x_1\otimes (\delta_{dend}f)(3,b_1,b_2,b_{3})\\
=\Psi(\delta_{dend}f)\big(x_1 \otimes b_1, x_2 \otimes b_2,x_{3}\otimes b_{3}\big)
\end{align*}

\textbf{General case:}
For a cochain $\Psi f$ of degree $n$, the Hochschild differential \(\delta_{HH}(\Psi f)\) is given by:
   \[
   \begin{aligned}
   \delta_{HH}(\Psi f)(x_1 \otimes b_1, \dots, x_{n+1} \otimes b_{n+1}) = (x_1 \otimes b_1) \cdot (\Psi f)(x_2 \otimes b_2, \dots, x_{n+1} \otimes b_{n+1})\\ 
   + \sum_{i=1}^n (-1)^i (\Psi f)(x_1 \otimes b_1, \dots, (x_i \otimes b_i) \cdot (x_{i+1} \otimes b_{i+1}), \dots, x_{n+1} \otimes b_{n+1}) \\
   + (-1)^{n+1} (\Psi f)(x_1 \otimes b_1, \dots, x_n \otimes b_n) \cdot (x_{n+1} \otimes b_{n+1})
   \end{aligned}
   \]

Let us compute each term in \(\delta_{HH}(\Psi f)\) one by one.

\textbf{1. First Term: \((x_1 \otimes b_1) \cdot (\Psi f)(x_2 \otimes b_2, \dots, x_{n+1} \otimes b_{n+1})\)}

Here \(\Psi f\) applied to \(n\) arguments:
\[
\begin{aligned}
\Psi f(x_2 \otimes b_2, \dots, x_{n+1} \otimes b_{n+1}) =\\ x_2 \cdots x_{n+1} \otimes f(1, b_2, \dots, b_{n+1}) 
+ x_3 x_2 x_4 \cdots x_{n+1} \otimes f(2, b_2, \dots, b_{n+1}) \\
+ \sum_{i=2}^{n-1} x_{i+1} \cdots x_{2 } x_{i+2} x_{i+3} \cdots x_{n+1} \otimes f(i+1, b_2, \dots, b_{n+1})\\ 
+ x_{n+1} \cdots x_2 \otimes f(n+1, b_2, \dots, b_{n+1})
\end{aligned}
\]

Now, multiply \((x_1 \otimes b_1)\) with each term in \(\Psi f\):

- For the first term in \(\Psi f\):
  \[
  \begin{aligned}
  (x_1 \otimes b_1) \cdot (x_2 \cdots x_{n+1} \otimes f(1, b_2, \dots, b_{n+1})) &= x_1 x_2 \cdots x_{n+1} \otimes (b_1 \prec f(1, b_2, \dots, b_{n+1})) \\
  &+ x_2 \cdots x_{n+1} x_1 \otimes (b_1 \succ f(1, b_2, \dots, b_{n+1}))
  \end{aligned}
  \]
  
- For the second term in \(\Psi f\):
  \begin{align*}
  (x_1 \otimes b_1) \cdot (x_3 x_2 x_4 \cdots x_{n+1} \otimes f(2, b_2, \dots, b_{n+1})) \\= x_1 x_3 x_2 x_4 \cdots x_{n+1} \otimes (b_1 \prec f(2, b_2, \dots, b_{n+1})) \\+ x_3 x_2 x_4 \cdots x_{n+1} x_1 \otimes (b_1 \succ f(2, b_2, \dots, b_{n+1}))
  \end{align*}

- For the \(i\)-th term in the sum:
  \[
  \begin{aligned}
  &(x_1 \otimes b_1) \cdot (x_{i+1} \cdots x_{2 } x_{i+2} x_{i+3} \cdots x_{n+1}  \otimes f(i+1, b_2, \dots, b_{n+1}))=\\ & x_1x_{i+1} \cdots x_{2 } x_{i+2} x_{i+3} \cdots x_{n+1}  \otimes (b_1 \prec f(i+1, b_2, \dots, b_{n+1}))\\& + x_{i+1} \cdots x_{2 } x_{i+2} x_{i+3} \cdots x_{n+1}  x_1 \otimes (b_1 \succ f(i+1, b_2, \dots, b_{n+1}))
   \end{aligned}
  \]
  
- For the last term in \(\Psi f\):
  \[
  \begin{aligned}
  (x_1 \otimes b_1) \cdot (x_{n+1} \cdots x_2 \otimes f(n, b_2, \dots, b_{n+1})) &= x_1 x_{n+1} \cdots x_2 \otimes (b_1 \prec f(n, b_2, \dots, b_{n+1}))\\ &+ x_{n+1} \cdots x_2 x_1 \otimes (b_1 \succ f(n, b_2, \dots, b_{n+1}))
  \end{aligned}
  \]

So, the first term in \(\delta_{HH}(\Psi f)\) is the sum of all these expressions.

\textbf{2. Middle Terms: \(\sum_{i=1}^n (-1)^i (\Psi f)(\dots, (x_i \otimes b_i) \cdot (x_{i+1} \otimes b_{i+1}), \dots)\)}

For each \(i\) from 1 to \(n\), we replace \((x_i \otimes b_i)\) and \((x_{i+1} \otimes b_{i+1})\) with their product:
\[
(x_i \otimes b_i) \cdot (x_{i+1} \otimes b_{i+1}) = x_i x_{i+1} \otimes (b_i \prec b_{i+1}) + x_{i+1} x_i \otimes (b_i \succ b_{i+1})
\]

 Apply \(\Psi f\),
for each \(i\), we get:
\[
\begin{aligned}
\Psi f(\dots, x_i x_{i+1} \otimes (b_i \prec b_{i+1}), \dots) = \\
   x_1 \cdots x_{i-1} (x_i x_{i+1}) x_{i+2} \cdots x_{n+1} \otimes f(1, b_1, \dots, b_{i-1}, b_i \prec b_{i+1}, b_{i+2}, \dots, b_{n+1})\\
   +x_2 x_1 \cdots x_{i-1} (x_i x_{i+1}) x_{i+2} \cdots x_{n+1} \otimes f(2, b_1, \dots, b_{i-1}, b_i \prec b_{i+1}, \dots, b_{n+1})\\
   +\sum_{k=2}^{n-1}x_{k+1} \cdots  x_{1} \cdots (x_i x_{i+1}) \cdots x_{n+1} \otimes f(k+1, \dots, b_i \prec b_{i+1}, \dots)\\
   +x_{n+1} \cdots (x_i x_{i+1}) \cdots x_1 \otimes f(n+1, \dots, b_i \prec b_{i+1}, \dots)
   \end{aligned}
   \]
   
\[
\begin{aligned}
\Psi f(\dots, x_{i+1} x_{i} \otimes (b_i \succ b_{i+1}), \dots) = \\
   (x_{i+1} x_{i})x_{i-1}\dots x_1 x_{i+2} \cdots x_{n+1} \otimes f(1, b_1, \dots, b_{i-1}, b_i \succ b_{i+1}, b_{i+2}, \dots, b_{n+1})\\
   +x_2 x_1 \cdots x_{i-1} (x_{i+1} x_{i}) x_{i+2} \cdots x_{n+1} \otimes f(2, b_1, \dots, b_{i-1}, b_i \succ b_{i+1}, \dots, b_{n+1})\\
   +\sum_{k=2}^{n-2}x_{k+1} \cdots  x_{1} \cdots (x_i x_{i+1}) \cdots x_{n+1} \otimes f(k+1, \dots, b_i \succ b_{i+1}, \dots)\\
   +x_{n+1} \cdots (x_{i+1} x_{i}) \cdots x_1 \otimes f(n, \dots, b_i \succ b_{i+1}, \dots)
   \end{aligned}
   \]

\textbf{3. Last Term: \((-1)^{n+1} (\Psi f)(x_1 \otimes b_1, \dots, x_n \otimes b_n) \cdot (x_{n+1} \otimes b_{n+1})\)}

We have:
\[
\begin{aligned}
\Psi f(x_1 \otimes b_1, \dots, x_n \otimes b_n) &= x_1 \cdots x_n \otimes f(1, b_1, \dots, b_n) + x_2 x_1 x_3 \cdots x_n \otimes f(2, b_1, \dots, b_n)\\ 
&+ \sum_{i=2}^{n-2} x_{i+1} \cdots x_{1} x_{i+2} x_{i+3} \cdots x_n \otimes f(i+1, b_1, \dots, b_n) 
\\ &+ x_n \cdots x_2 x_1 \otimes f(n, b_1, \dots, b_n)
\end{aligned}
\]

Now, multiply each term by \((x_{n+1} \otimes b_{n+1})\):

- For the first term:
  \[
  \begin{aligned}
  (x_1 \cdots x_n \otimes f(1, b_1, \dots, b_n)) \cdot (x_{n+1} \otimes b_{n+1}) \\= x_1 \cdots x_n x_{n+1} \otimes (f(1, b_1, \dots, b_n) \prec b_{n+1}) \\+ x_{n+1} x_1 \cdots x_n \otimes (f(1, b_1, \dots, b_n) \succ b_{n+1})
  \end{aligned}
  \]
  
- For the second term:
  \[
  \begin{aligned}
  (x_2 x_1 x_3 \cdots x_n \otimes f(2, b_1, \dots, b_n)) \cdot (x_{n+1} \otimes b_{n+1}) \\= x_2 x_1 x_3 \cdots x_n x_{n+1} \otimes (f(2, b_1, \dots, b_n) \prec b_{n+1})\\ + x_{n+1} x_2 x_1 x_3 \cdots x_n \otimes (f(2, b_1, \dots, b_n) \succ b_{n+1})
  \end{aligned}
  \]
  
- For the \(i\)-th term in the sum:
  \[
  \begin{aligned}
  &(x_{i+1} \cdots x_{1} x_{i+2} x_{i+3} \cdots x_n \otimes f(i+1, b_1, \dots, b_n)) \cdot (x_{n+1} \otimes b_{n+1}) \\&= x_{i+1} \cdots x_{1} x_{i+2} x_{i+3} \cdots x_n x_{n+1} \otimes (f(i+1, b_1, \dots, b_n) \prec b_{n+1})\\& + x_{n+1}x_{i+1} \cdots x_{1} x_{i+2} x_{i+3} \cdots x_n \otimes (f(i+1, b_1, \dots, b_n) \succ b_{n+1})
   \end{aligned}
  \]
  
- For the last term:

\begin{longtable}{@{}l@{}l@{}}

$\displaystyle (x_n \cdots x_2 x_1 \otimes f(n, b_1, \dots, b_n)) \cdot (x_{n+1} \otimes b_{n+1})$
& $ =\displaystyle x_n \cdots x_2 x_1 x_{n+1} \otimes (f(n, b_1, \dots, b_n) \prec b_{n+1})$ \\[4pt]
& $\displaystyle + x_{n+1} x_n \cdots x_2 x_1 \otimes (f(n, b_1, \dots, b_n) \succ b_{n+1})$
\end{longtable}

The last term in \(\delta_{HH}(\Psi f)\) is \((-1)^{n+1}\) times the sum of all these expressions.

Alignment of Terms:
\begin{itemize}
    \item First terms corresponds to $b_1\succ f(i,\dots,)$ or $b_1\prec f(i,\dots,)$ in $\delta_{dend}$,
\item Middle Terms  corresponds to $b_i\prec b_{i+1}+b_i\succ b_{i+1}$  in $\delta_{dend}$,
\item last Terms corresponds to $f(i,\dots,)\succ b_{n+1}$ or $f(i,\dots,)\prec b_{n+1}$ in $\delta_{dend}$
    \end{itemize}
By the Perm identity
\[
x_ix_{\sigma{(1)}}\dots \hat{x}_{\sigma{(i)}}\dots x_{\sigma{(n+1)}}=x_ix_1\dots \hat{x}_{i}\dots x_{n+1};\quad \ for  \quad i=1,\dots n+1,\quad \sigma \in S_{n+1}.
\]
we get:
\begin{longtable}{@{}l@{}l@{}}
$\displaystyle \delta_{HH}(\Psi f)$ 
& $(x_1 \otimes b_1, \dots,\displaystyle x_{n+1} \otimes b_{n+1}) = x_1x_2\dots x_{n+1}\otimes \sum_{i=1}^n a_1\prec f(i,b_2,\dots,b_{n+1})$ \\[4pt]
& $\displaystyle - x_1x_2\dots x_{n+1}\otimes f(1,b_1\succ b_2,\dots,b_{n+1})$ \\[4pt]
& $\displaystyle + x_1x_2\dots x_{n+1}\otimes\sum_{i=2}^n(-1)^i f(1,b_1,\dots, b_i*b_{i+1},\dots,b_{n+1})$ \\[4pt]
& $\displaystyle + (-1)^{n+1} x_1x_2\dots x_n\otimes f(1,b_1,b_2,\dots,b_{n})\succ b_{n+1}$ \\[4pt]
& $\displaystyle + x_2x_1\dots x_{n+1}\otimes b_1\succ f(1,b_2,\dots,b_{n+1})$ \\[4pt]
& $\displaystyle - x_2x_1\dots x_{n+1}\otimes f(1,b_1\succ b_2,b_3,\dots,b_{n+1})$ \\[4pt]
& $\displaystyle + x_2x_1\dots x_{n+1}\otimes f(2,b_1,b_{2}\prec b_{3},\dots,b_{n+1})$ \\[4pt]
& $\displaystyle + x_2x_1\dots x_{n+1}\otimes \sum_{i=3}^{n}(-1)^i f(2,b_1,\dots, b_i*b_{i+1},\dots,b_{n+1})$ \\[4pt]
& $\displaystyle + (-1)^{n+1}x_2x_1\dots x_{n+1}\otimes f(2,b_1,\dots,b_n)\prec b_{n+1}$ \\[4pt]
& $\displaystyle + \sum_{i=2}^{n-1}x_{i+1}x_i\dots x_1x_{i+2}\dots x_{n+1}\otimes \sum_{i=1}^{r-1}b_1\succ f(r-1,b_2,\dots,b_{n+1})$ \\[4pt]
& $\displaystyle + \sum_{i=2}^{n-1}x_{i+1}x_i\dots x_1x_{i+2}\dots x_{n+1}\otimes \sum_{i=1}^{r-2}(-1)^i f(r-1,b_1,\dots, b_i*b_{i+1},\dots,b_{n+1})$ \\[4pt]
& $\displaystyle + (-1)^{r-1}\sum_{i=2}^{n-1}x_{i+1}x_i\dots x_1x_{i+2}\dots x_{n+1}\otimes f(r-1,b_1,\dots,b_{r-1}\succ b_r,\dots,b_{n+1})$ \\[4pt]
& $\displaystyle + (-1)^{r}\sum_{i=2}^{n-1}x_{i+1}x_i\dots x_1x_{i+2}\dots x_{n+1}\otimes f(r,b_1,\dots,b_{r}\prec b_{r+1},\dots,b_{n+1})$ \\[4pt]
& $\displaystyle + \sum_{i=2}^{n-1}x_{i+1}x_i\dots x_1x_{i+2}\dots x_{n+1}\otimes \sum_{i=r+1}^{n}(-1)^i f(r,b_1,\dots, b_i*b_{i+1},\dots,b_{n+1})$ \\[4pt]
& $\displaystyle + \sum_{i=2}^{n-1}x_{i+1}x_i\dots x_1x_{i+2}\dots x_{n+1}\otimes \sum_{i=r+1}^{n}(-1)^{n+1} f(r,b_1,\dots,b_{n})\prec b_{n+1}$ \\[4pt]
& $\displaystyle + x_{n+1}x_{n}\dots x_2x_1\otimes b_1\succ f(n,b_2,\dots,b_{n+1})$ \\[4pt]
& $\displaystyle + \sum_{i=1}^{n-1}(-1)^i x_{n+1}x_{n}\dots x_2x_1\otimes f(n,b_1,\dots, b_i*b_{i+1},\dots,b_{n+1})$ \\[4pt]
& $\displaystyle + (-1)^{n}x_{n+1}x_{n}\dots x_2x_1\otimes f(n,b_1,b_2,\dots,b_n\succ b_{n+1})$ \\[4pt]
& $\displaystyle + \sum_{i=1}^n(-1)^{n+1}x_{n+1}x_{n}\dots x_2x_1\otimes f(i,b_1,b_2,\dots,b_n)\succ b_{n+1}$ \\[4pt]
& $\displaystyle = x_1x_2\dots x_{n} \otimes (\delta_{dend}f)(1,b_1 , b_2,\dots,b_{n+1})$ \\[4pt]
& $\displaystyle + x_2x_1x_3\dots x_{n+1} \otimes (\delta_{dend}f)(2,b_1, b_2,\dots,b_{n+1})$ \\[4pt]
& $\displaystyle + \sum_{i=2}^{n-1}x_{i+1}x_i\dots x_1x_{i+2}\dots x_{n+1}\otimes (\delta_{dend}f)(i+1,b_1,b_2,\dots,b_{n+1})$ \\[4pt]
& $\displaystyle + x_{n+1}x_{n}\dots x_2x_1\otimes (\delta_{dend}f)(n+1,b_1,b_2,\dots,b_{n+1})$ \\[4pt]
& $\displaystyle = \Psi(\delta_{dend}(f))\big(x_1 \otimes b_1, x_2 \otimes b_2,\dots,x_{n+1}\otimes b_{n+1}\big)$ \\
\end{longtable}

   \end{proof}
  \begin{remark}\label{re:3.5}
In the case where the Perm algebra \(A\) is free, the cochain map \(\Psi\) is injective. 
(By ``free Perm algebra'' we mean \(A \cong \mathrm{Perm}\langle X \rangle\) for some set \(X\), 
where \(\mathrm{Perm}\langle X \rangle\) is constructed as follows.

Let \(F\langle X \rangle\) denote the free non-unital associative algebra over \(X = \{x_1, x_2, \dots\}\). 
Consider the two-sided ideal \(I \subset F\langle X \rangle\) generated by all elements of the form
\[
x(yz) - x(zy), \quad \text{for all } x, y, z \in F\langle X \rangle.
\]
Then define
\[
\mathrm{Perm}\langle X \rangle := F\langle X \rangle / I.
\]
By construction, \(\mathrm{Perm}\langle X \rangle\) satisfies the Perm identity 
\((a \cdot b) \cdot c = a \cdot (b \cdot c) = a \cdot (c \cdot b)\) for all \(a, b, c \in \mathrm{Perm}\langle X \rangle\), 
and it is free in the sense that any map \(X \to P\) to a Perm algebra \(P\) extends uniquely 
to a Perm algebra homomorphism \(\mathrm{Perm}\langle X \rangle \to P\)).
\end{remark}
\begin{corollary} 
    Given the canonical embedding of cochain complexes
\[
C^*_{\mathrm{dend}}(B,N) \hookrightarrow C^*_{\mathrm{Hoch}}(A\otimes B,A\otimes N),
\]
we obtain a short exact sequence of complexes:
\[
0 \to C^*_{\mathrm{dend}}(B,N) \hookrightarrow C^*_{\mathrm{Hoch}}(A\otimes B,A\otimes N)\to Q^* \to 0,
\]
where $A$ is a perm-free algebra, $B$ is a pre-associative algebra, $N$ is a $B$-bimodule, and $ Q^* = C^*_{\mathrm{Hoch}}(A\otimes B,A\otimes N)/C^*_{\mathrm{dend}}(B,N)$ is the quotient complex.
Applying the cohomology functor yields the long exact sequence:
\[
\begin{aligned}
\cdots \to H^{n-1}_{\mathrm{dend}}(B,N) &\to HH^{n-1}(A\otimes B,A\otimes N) \to H^{n-1}(Q^*) \\
&\to H^n_{\mathrm{dend}}(B,N) \to HH^n(A\otimes B,A\otimes N) \to H^n(Q^*) \\
&\to H^{n+1}_{dend}(B,N) \to HH^{n+1}(A\otimes B,A\otimes N) \to \cdots
\end{aligned}
\]
\end{corollary}

\begin{example}\label{exm:5.7}
Let:
\begin{itemize}
\item $A = F\langle x_1,x_2,x_3\rangle$ be the free Perm algebra.
\item $B$ is the pre-associative algebra defined in \ref{exp:dend}.
\end{itemize}
\end{example}
A Hochschild 2-cochain \( \Psi: (A \otimes B) \otimes (A \otimes B) \to (A \otimes B) \) can be expressed in terms of \( f \) as:
\[
\Psi(x_1 \otimes e_i, x_2 \otimes e_j) =  x_1x_2 \otimes f(1, e_i, e_j) + x_2x_1 \otimes f(2, e_i, e_j)=\sum_{k=1}^n x_1x_2 \otimes \alpha^k_{ij} e_k + x_2x_1 \otimes \beta^k_{ij} e_k.
\]

To compute the Hochschild differential \( \delta_{HH} \Psi \) evaluated at \( (x_1 \otimes e_1, x_2 \otimes e_1, x_3 \otimes e_1) \), we'll follow the standard definition of the Hochschild differential for a cochain \( \Psi \):
\[
(\delta_{HH} \Psi)(a, b, c) = a \cdot \Psi(b, c) - \Psi(a \cdot b, c) + \Psi(a, b \cdot c) - \Psi(a, b) \cdot c
\]
where \( a = x_1 \otimes e_1 \), \( b = x_2 \otimes e_1 \), and \( c = x_3 \otimes e_1 \).
\begin{itemize}
    \item 
  Compute \( \Psi(b, c) = \Psi(x_2 \otimes e_1, x_3 \otimes e_1) \):
\[
\Psi(x_2 \otimes e_1, x_3 \otimes e_1) = \sum_{k=1}^n x_2x_3 \otimes \alpha^k_{11} e_k + \sum_{k=1}^n x_3x_2 \otimes \beta^k_{11} e_k
\]

\item Compute \( a \cdot \Psi(b, c) \)

Using the product in \( A \otimes B \):
\[
(x_1 \otimes e_1) \cdot (x_2x_3 \otimes \alpha^k_{11} e_k) = x_1(x_2x_3) \otimes (e_1 \prec \alpha^k_{11} e_k) + (x_2x_3)x_1 \otimes (e_1 \succ \alpha^k_{11} e_k).
\]
Since \(e_1 \prec e_k = 0\) for all \(k \geq 1\) and \(e_1 \succ e_k = e_1\) only if \(k = 1\), this simplifies to:
\[
(x_2x_3)x_1 \otimes \alpha^1_{11} e_1.
\]
Similarly:
\[
(x_1 \otimes e_1) \cdot (x_3x_2 \otimes \beta^k_{11} e_k) = (x_3x_2)x_1 \otimes \beta^1_{11} e_1.
\]
Thus, the first term is:
\[
(x_2x_3)x_1 \otimes \alpha^1_{11} e_1 + (x_3x_2)x_1 \otimes \beta^1_{11} e_1.
\]

 \item  Compute \(-\Psi((x_1 \otimes e_1) \cdot (x_2 \otimes e_1), x_3 \otimes e_1)\):
\[
(x_1 \otimes e_1) \cdot (x_2 \otimes e_1)=x_1x_2 \otimes (e_1 \prec e_1) + x_2x_1 \otimes (e_1 \succ e_1) = x_2x_1 \otimes e_1.
\]
Now apply \(\Psi\):
\[
\Psi(x_2x_1 \otimes e_1, x_3 \otimes e_1) = \sum_{k=1}^n (x_2x_1)x_3 \otimes \alpha^k_{11} e_k + \sum_{k=1}^n x_3(x_2x_1) \otimes \beta^k_{11} e_k.
\]

 \item Compute \(\Psi(x_1 \otimes e_1, (x_2 \otimes e_1) \cdot (x_3 \otimes e_1))\):
\[
(x_2 \otimes e_1) \cdot (x_3 \otimes e_1)=x_2x_3 \otimes (e_1 \prec e_1) + x_3x_2 \otimes (e_1 \succ e_1) = x_3x_2 \otimes e_1.
\]
Now apply \(\Psi\):
\[
\Psi(x_1 \otimes e_1, x_3x_2 \otimes e_1) = \sum_{k=1}^n x_1(x_3x_2) \otimes \alpha^k_{11} e_k + \sum_{k=1}^n (x_3x_2)x_1 \otimes \beta^k_{11} e_k.
\]

 \item  Compute \(-\Psi(x_1 \otimes e_1, x_2 \otimes e_1) \cdot (x_3 \otimes e_1)\)
\[
\Psi(x_1 \otimes e_1, x_2 \otimes e_1)=\sum_{k=1}^n x_1x_2 \otimes \alpha^k_{11} e_k + \sum_{k=1}^n x_2x_1 \otimes \beta^k_{11} e_k.
\]
Multiply by \((x_3 \otimes e_1)\):
\[
\begin{aligned}
&(x_1x_2)x_3 \otimes (\alpha^k_{11} e_k \prec e_1) + x_3(x_1x_2) \otimes (\alpha^k_{11} e_k \succ e_1)\\
&+ (x_2x_1)x_3 \otimes (\beta^k_{11} e_k \prec e_1) + x_3(x_2x_1) \otimes (\beta^k_{11} e_k \succ e_1)\\
\end{aligned}
\]
Thus, the fourth term is:
\begin{longtable}{@{}c@{}c@{}}
& $\displaystyle \sum_{k=2}^n (x_1x_2)x_3 \otimes \alpha^k_{11} e_k + x_3(x_1x_2) \otimes \alpha^1_{11} e_1\displaystyle + \sum_{k=2}^n (x_2x_1)x_3 \otimes \beta^k_{11} e_k + x_3(x_2x_1) \otimes \beta^1_{11} e_1.$
\end{longtable}

 \item  Combine All Terms

Now, substitute all computed terms back into the Hochschild differential:

\[
\begin{aligned}
\delta_{HH} \Psi(a, b, c) = &(x_2x_3)x_1 \otimes \alpha^1_{11} e_1 + (x_3x_2)x_1 \otimes \beta^1_{11} e_1\\
&-\sum_{k=1}^n (x_2x_1)x_3 \otimes \alpha^k_{11} e_k - \sum_{k=1}^n x_3(x_2x_1) \otimes \beta^k_{11} e_k\\
&+ \sum_{k=1}^n x_1(x_3x_2) \otimes \alpha^k_{11} e_k + \sum_{k=1}^n (x_3x_2)x_1 \otimes \beta^k_{11} e_k\\
&-\sum_{k=2}^n(x_1x_2)x_3 \otimes \alpha^k_{11} e_k  - x_3(x_1x_2) \otimes \alpha^1_{11} e_1 \\
&-\sum_{k=2}^n (x_2x_1)x_3 \otimes \beta^k_{11} e_k - x_3(x_2x_1) \otimes \beta^1_{11} e_1
\end{aligned}
\]

\item Using the Perm identities, the Hochschild differential \( \delta_{HH} \Psi \) evaluated on the triple \( (x_1 \otimes e_1, x_2 \otimes e_1, x_3 \otimes e_1) \) is:

\[
\begin{aligned}
x_1x_2x_3\otimes\alpha^1_{11}e_1-\sum_{k=2}^n x_2x_1x_3 \otimes\big( \alpha^k_{11} +\beta^k_{11}\big) e_k- x_3x_1x_2 \otimes \alpha^1_{11} e_1\\
\end{aligned}
\]
Similarly:

\item The Hochschild differential \( \delta_{HH} \Psi \) evaluated at \( (x_1 \otimes e_i, x_2 \otimes e_1, x_3 \otimes e_1) \) is:
\begin{longtable}{@{}l@{}l@{}}
& $\displaystyle x_1 x_2 x_3 \otimes \big((\alpha_{11}^1 + \beta_{11}^1)e_1 - \sum_{k=2}^{n} \alpha_{i1}^k e_k \big)\displaystyle - x_2 x_1x_3 \otimes \sum_{k=2}^n \beta_{i1}^k e_k - x_3 x_1x_2 \otimes (\alpha^1_{11} + \beta^1_{11})e_1$
\end{longtable}

\item The Hochschild differential \( \delta_{HH} \Psi \) evaluated at \( (x_1 \otimes e_1, x_2 \otimes e_i, x_3 \otimes e_1) \) is:
$$(x_1 x_2) x_3 \otimes \alpha_{1i}^1 e_1 + (x_2 x_1) x_3 \otimes (\alpha_{i1}^1 + \beta_{1i}^1) e_1  + (x_3 x_1) x_2 \otimes (- \alpha_{1i}^1+\beta_{i1}^1 - \beta_{1i}^1) e_1
$$

\item The Hochschild differential \( \delta_{HH} \Psi \) evaluated at \( (x_1 \otimes e_1, x_2 \otimes e_1, x_3 \otimes e_i) \) is:

\[
\begin{aligned}
& - x_2 x_1 x_3 \otimes \sum_{k=2}^n\alpha^k_{1i}e_k  + x_3 x_1 x_2 \otimes \beta_{1i}^1 e_1
\end{aligned}
\]

\item The Hochschild differential \( \delta_{HH} \Psi \) evaluated at \( (x_1 \otimes e_i, x_2 \otimes e_j, x_3 \otimes e_1) \) is:
\[
\begin{aligned}
& x_1 x_2 x_3 \otimes ( \alpha_{j1}^1+\beta_{j1}^1 +\alpha^1_{ij})e_1  + x_2 x_1 x_3 \otimes \beta_{ij}^1 e_1  - x_3 x_1 x_2 \otimes(\alpha^1_{ij}+\beta^1_{ij})e_1
\end{aligned}
\]

\item The Hochschild differential \( \delta_{HH} \Psi \) evaluated at \( (x_1 \otimes e_i, x_2 \otimes e_1, x_3 \otimes e_j) \) is:

\[
\begin{aligned}
& (x_1 x_2) x_3 \otimes ((\alpha_{ij}^1 +\beta_{ij}^1)e_1  - \sum_{k=1}^n\alpha^k_{ij}e_k)  - (x_3 x_1) x_2 \otimes \sum_{k=1}^n\beta^k_{ij}e_k
\end{aligned}
\]

\item The Hochschild differential \( \delta_{HH} \Psi \) evaluated at \( (x_1 \otimes e_1, x_2 \otimes e_i, x_3 \otimes e_j) \) is:

\[
\begin{aligned}
  x_2 x_1 x_3 \otimes \alpha_{ij}^1e_1+ x_3 x_1 x_2 \otimes \beta_{ij}^1e_1
\end{aligned}
\]

\item The Hochschild differential \( \delta_{HH} \Psi \) evaluated at \( (x_1 \otimes e_i, x_2 \otimes e_j, x_3 \otimes e_k) \) is:

\[
\begin{aligned}
& (x_1 x_2) x_3 \otimes (\alpha^1_{jk}+\beta^1_{jk})e_1  + (x_2 x_1) x_3 \otimes \alpha_{ij}^1e_1  + (x_3 x_1) x_2 \otimes \beta_{ij}^1e_1
\end{aligned}
\]
\end{itemize}
For $\delta_{HH}\Psi=0$, we have: 
 
\[
\begin{aligned}
\alpha^1_{11}&=0,\\
\alpha^k_{1j}&=0,\  \forall j\geq 2,\ k\geq 1\\
\alpha^k_{i1}&=0,\  \forall i\geq 2, k\geq 1\\
\alpha^k_{ij}&=0,\  \forall i,j,k\geq 2\\
\beta^1_{11}&=0,\\
\beta^k_{1j}&=0,\  \forall j\geq 2,\ k\geq 1\\
\beta^k_{i1}&=0,\  \forall i\geq 2, k\geq 1\\
\beta^k_{ij}&=0,\  \forall i,j,k\geq 2\\
\alpha^k_{11}&+\beta^k_{11}=0,\  \forall  k\geq 2.
\end{aligned}
\]

The  solution satisfying all cochains conditions is:
\[
{(\Psi f)(x_1\otimes e_1, x_1\otimes e_1) = x_1x_2 \otimes \sum_{k=2}^n\alpha_{11}^k e_k - x_2x_1 \otimes \sum_{k=2}^n\alpha_{11}^k e_k \quad }
\]
To compute the Hochschild coboundary \(\delta_{HH} (\Psi g)\) evaluated at \((x_1 \otimes e_1, x_2 \otimes e_1)\), we use the Hochschild differential formula for a 1-cochain \(\Psi g\):

\[
(\delta_{HH} (\Psi g))(a, b) = a \cdot (\Psi g)(b) - (\Psi g)(a \cdot b) + (\Psi g)(a) \cdot b.
\]

1. First term: \( (x_1 \otimes e_1) \cdot (\Psi g)(x_2 \otimes e_1) \):
     \[
     (\Psi g)(x_2 \otimes e_1) = x_2 \otimes \sum_{k=1}^n c_1^k e_k.
     \]
   - Now multiply by \((x_1 \otimes e_1)\) using the product in \(A \otimes B\):
     \[
     (x_1 \otimes e_1) \cdot \left( x_2 \otimes \sum_{k=1}^n c_1^k e_k \right) = x_1x_2 \otimes \left( e_1 \prec \sum_{k=1}^n c_1^k e_k \right) + x_2x_1 \otimes \left( e_1 \succ \sum_{k=1}^n c_1^k e_k \right).
     \]
   - Simplify using the structure of \(B\):
     \[
     e_1 \prec e_k = 0 \quad \text{for all } k, \quad e_1 \succ e_k = e_1 \text{ only if } k = 1.
     \]
     Thus:
     \[
     = x_2x_1 \otimes c_1^1 e_1.
     \]

2. Second term: \( -(\Psi g)((x_1 \otimes e_1) \cdot (x_2 \otimes e_1)) \):
     \[
     x_1x_2 \otimes (e_1 \prec e_1) + x_2x_1 \otimes (e_1 \succ e_1) = x_2x_1 \otimes e_1.
     \]
   - Apply \(\Psi g\):
     \[
     (\Psi g)(x_2x_1 \otimes e_1) = x_2x_1 \otimes \sum_{k=1}^n c_1^k e_k.
     \]

3. Third term: \( (\Psi g)(x_1 \otimes e_1) \cdot (x_2 \otimes e_1) \):
     \[
     (\Psi g)(x_1 \otimes e_1) = x_1 \otimes \sum_{k=1}^n c_1^k e_k.
     \]
   - Multiply by \((x_2 \otimes e_1)\):
     \[
     \begin{aligned}
     \left( x_1 \otimes \sum_{k=1}^n c_1^k e_k \right) \cdot (x_2 \otimes e_1)& = x_1x_2 \otimes \left( \sum_{k=1}^n c_1^k e_k \prec e_1 \right) + x_2x_1 \otimes \left( \sum_{k=1}^n c_1^k e_k \succ e_1 \right)\\
     &=x_1x_2 \otimes  \sum_{k=2}^n c_1^k e_k  + x_2x_1 \otimes  c_1^1e_1
      \end{aligned}.
     \]

{Combining all terms}
\[
\begin{aligned}
(\delta_{HH} (\Psi g))(x_1 \otimes e_1, x_2 \otimes e_1) = x_1x_2 \otimes \sum_{k=2}^n c_1^k e_k - x_2x_1\otimes \sum_{k=1}^n c_1^k e_k.
\end{aligned}
\]

It is the same result as in Example \ref{2.7}.
\section{pre-Lie algebras}
\begin{definition}[\!\!\cite{Gerstenhaber}]
A pre-Lie algebra $(P, \cdot_P)$ is a vector space  with a binary operation $\cdot_P : P \otimes P \rightarrow P$ such that
\[
(a  \cdot_P b) \cdot_P  c - a \cdot_P (b \cdot_P c) = (a \cdot_P c) \cdot_P b - a \cdot_P (c \cdot_P b), \quad for \  a, b, c \in P.
\]
The \textbf{total product} $[a,b] := a \cdot_P b - b \cdot_P a$ is the \textbf{Lie bracket} induced by the pre-Lie product.
\end{definition}
\begin{definition}[\!\!\cite{Gerstenhaber}]
A bimodule over a pre-Lie algebra  $(P, \cdot)$ is a vector space  M  with left and right actions, satisfying:
\[
\begin{aligned}
& (a \cdot b) \cdot m - a \cdot (b \cdot m) = (a \cdot m) \cdot b - a \cdot (m \cdot b), \\
& (m \cdot a) \cdot b - m \cdot (a \cdot b) = (m \cdot b) \cdot a - m \cdot (b \cdot a).
\end{aligned}
\]
for $a,b\in P$, $m\in M.$
\end{definition}
\begin{example}
    Define a \textbf{pre-Lie algebra} \((P, \cdot_P)\) generated by \( e_1, e_2 \), and product:
\[
e_1 \cdot_P e_2 = e_1, \quad \text{all other products zero}.
\]
\textbf{Lie Bracket} $[e_1,e_2]=e_1$ and all other products zero.
\end{example}
\begin{definition}[\!\!\cite{Dzhumadil’daev}]
We consider a pre-Lie algebra \((P, \cdot_P)\) and a \textbf{representation} \(N\) of \(P\). The \textbf{cochain complex} for right-symmetric algebra cohomology is given by:
\[
C_{\text{pre}}^{*}(P, N) = \bigoplus_{k \succ  0} C_{\text{pre}}^{k}(P, N),
\]
where:
\begin{itemize}
    \item \(C_{\text{pre}}^{k}(P, N)\) consists of multilinear maps \(\psi: P^{\otimes k} \to N\),
    \item The \textbf{coboundary operator} \(\delta_{\text{pre}}: C_{\text{pre}}^{k} \to C_{\text{pre}}^{k+1}\) is defined as:
\end{itemize}

\[
\begin{aligned}
(\delta_{\text{pre}} \psi)(a_0, a_1, \dots, a_k) &= -\sum_{i=1}^{k} (-1)^i a_0 .\psi(a_i, a_1, \dots, \hat{a}_i, \dots, a_k) \\
&+ \sum_{i=1}^{k} (-1)^i \psi(a_0 \cdot_P a_i, a_1, \dots, \hat{a}_i, \dots, a_k) \\
&+ \sum_{1 \leq i \prec  j \leq k} (-1)^{i+1} \psi(a_0, a_1, \dots, \hat{a}_i, \dots, [a_i, a_j], \dots, a_k) \\
&- \sum_{i=1}^{k} (-1)^i \psi(a_0, a_1, \dots, \hat{a}_i, \dots, a_k) . a_i,
\end{aligned}
\]

where:
\begin{itemize}
    \item \(\hat{a}_i\) means \(a_i\) is omitted,
    \item \([a_i, a_j] = a_i \cdot_P a_j - a_j \cdot_P a_i\) is the \textbf{Lie bracket} induced by the pre-Lie product.
\end{itemize}

\begin{enumerate}
    \item \textbf{Degree 1 (\(k=1\)):}
    \[
        (\delta_{\text{pre}} \psi)(a_0, a_1) = -a_0 \cdot \psi(a_1) + \psi(a_0 \cdot_p a_1) - \psi(a_0) \cdot a_1.
        \]

    \item \textbf{Degree 2 (\(k=2\)):}
   
        \[
        \begin{aligned}
        (\delta_{\text{pre}} \psi)(a_0, a_1, a_2) &= -a_0 \cdot \psi(a_1, a_2) +a_0 \cdot \psi(a_2, a_1)+ \psi(a_0 \cdot_P a_1, a_2) - \psi(a_0 \cdot_P a_2, a_1) \\
        &+ \psi(a_0, [a_1, a_2]) - \psi(a_0, a_1) \cdot a_2 + \psi(a_0, a_2) \cdot a_1.
        \end{aligned}
        \]
        
 \item \textbf{Degree 3 (\(k=3\)):}
\[
\begin{aligned}
(\delta_{\text{pre}} \psi)(a_0, a_1, a_2, a_3) &= -a_0 \cdot \psi(a_1, a_2, a_3) + a_0 \cdot \psi(a_2, a_1, a_3) \\
&- a_0 \cdot \psi(a_3, a_1, a_2) + \psi(a_0 \cdot_P a_1, a_2, a_3) \\
&- \psi(a_0 \cdot_P a_2, a_1, a_3) + \psi(a_0 \cdot_P a_3, a_1, a_2) \\
&+ \psi(a_0, [a_1, a_2], a_3) - \psi(a_0, [a_1, a_3], a_2) \\
&+ \psi(a_0, [a_2, a_3], a_1) - \psi(a_0, a_1, a_2) \cdot a_3 \\
&+ \psi(a_0, a_1, a_3) \cdot a_2 - \psi(a_0, a_2, a_3) \cdot a_1.
\end{aligned}
\]
\end{enumerate}
\end{definition}
\begin{example}\label{exm:5.4}
We find $H^n_{pre}(P,P)$, where $P$ is a pre-Lie generated by  \( e_1, e_2 \) and product:
    \[
    e_1 \cdot_P e_2 = e_1, \quad \text{all other products } e_i \cdot_P e_j = 0.
    \]
 \textbf{Lie bracket}:
    \[
    [e_1, e_2] = e_1 \cdot_P e_2 - e_2 \cdot_P e_1 = e_1 - 0 = e_1.
    \]
\begin{itemize}
    \item \textbf{2-cochains}: Bilinear maps \( \phi: P \times P \to P \).
    \item \textbf{Coboundary}:
    \[
    \begin{aligned}
        (\delta_{\text{pre}} \phi)(x, y, z) =& -x \cdot_P \phi(y, z) +x \cdot_P \phi(z, y)+ \phi(x \cdot_P y, z) - \phi(x \cdot_P z, y) \\ &+ \phi(x, [y, z]) - \phi(x, y) \cdot_P z + \phi(x, z) \cdot_P y.
    \end{aligned}
    \]
    \item \textbf{2-cocycles}: Maps \( \phi \) satisfying \( \delta_{\text{pre}} \phi = 0 \). Assume:
    \[
    \phi(e_1, e_2) = a e_1 + b e_2, \quad \phi(e_2, e_1) = c e_1 + d e_2, \quad \phi(e_i, e_i) = 0.
    \]

 For \( (x, y, z) = (e_1, e_2, e_1) \):
\[
\begin{aligned}
(\delta_{\text{pre}} \phi)(e_1, e_2, e_1) &= -e_1 \cdot_P \phi(e_2, e_1) + e_1 \cdot_P \phi(e_1, e_2) + \phi(e_1 \cdot_P e_2, e_1) \\&- \phi(e_1 \cdot_P e_1, e_2) + \phi(e_1, [e_2, e_1]) - \phi(e_1, e_2) \cdot_P e_1 + \phi(e_1, e_1) \cdot_P e_2.\\
&= -e_1 \cdot_P (c e_1 + d e_2) + e_1 \cdot_P (a e_1 + b e_2) + \phi(e_1, e_1)\\ & - \phi(0, e_2) + \phi(e_1, -e_1) - (a e_1 + b e_2) \cdot_P e_1.
\\
&= -c (e_1 \cdot_P e_1) - d (e_1 \cdot_P e_2) + a (e_1 \cdot_P e_1) + b (e_1 \cdot_P e_2)  - a (e_1 \cdot_P e_1)\\ & - b (e_2 \cdot_P e_1)= -d e_1 + b e_1 = (-d + b) e_1.
\end{aligned}
\]
For \( \delta_{\text{pre}} \phi = 0 \), we must have:
\[
-d + b = 0 \implies b = d.
\]

Similarly, evaluating \( \delta_{\text{pre}} \phi \) for \( (e_2, e_1, e_2) \):
\[
(\delta_{\text{pre}} \phi)(e_2, e_1, e_2) = de_2
\]
For \( \delta_{\text{pre}} \phi = 0 \), we must have:
\[
d = 0 \implies b = d = 0.
\]

From the above, a 2-cocycle \( \phi \) must satisfy:
\[
b = d = 0, \quad \phi(e_1, e_2) = a e_1, \quad \phi(e_2, e_1) = c e_1.
\]
Additionally, the cocycle condition imposes \( \phi(e_i, e_i) = 0 \).

A 2-cochain \( \phi \) is a coboundary if \( \phi = \delta_{\text{pre}} \psi \) for some 1-cochain \( \psi \). The coboundary operator for a 1-cochain \( \psi \) is:
\[
(\delta_{\text{pre}} \psi)(x, y) = -x \cdot_P \psi(y) + \psi(x \cdot_P y) - \psi(x) \cdot_P y.
\]
Let \( \psi(e_1) = \alpha e_1 + \beta e_2 \) and \( \psi(e_2) = \gamma e_1 + \delta e_2 \). Then:
\[
\phi(e_1, e_2) = (\delta_{\text{pre}} \psi)(e_1, e_2) = -e_1 \cdot_P \psi(e_2) + \psi(e_1 \cdot_P e_2) - \psi(e_1) \cdot_P e_2.
\]
Substituting:
\[
= -e_1 \cdot_P (\gamma e_1 + \delta e_2) + \psi(e_1) - (\alpha e_1 + \beta e_2) \cdot_P e_2 = -\delta e_1 + (\alpha e_1 + \beta e_2) - \alpha e_1 = -\delta e_1 + \beta e_2.
\]
Similarly:
\[
\phi(e_2, e_1) = (\delta_{\text{pre}} \psi)(e_2, e_1) = -e_2 \cdot_P \psi(e_1) + \psi(e_2 \cdot_P e_1) - \psi(e_2) \cdot_P e_1 = 0.
\]
Thus, a coboundary \( \phi \) must satisfy:
\[
\phi(e_1, e_2) = -\delta e_1 + \beta e_2, \quad \phi(e_2, e_1) = 0.
\]
Comparing with the general 2-cocycle:
\[
\phi(e_1, e_2) = a e_1, \quad \phi(e_2, e_1) = c e_1.
\]
For \( \phi \) to be a coboundary, we must have \( c = 0 \) and \( a = -\delta \), \( \beta = 0 \). Therefore, the cohomology is determined by the parameter \( c \), which is not constrained by the coboundary condition. Therefore:
\[
H^2_{\text{pre}}(P, P) \cong \Bbbk.
\]
\end{itemize}
\end{example}

\section{Lie and pre-Lie algebras}
\begin{theorem}[\!\!{\cite[Example 3.2]{GubarevKolesnikov}}]\label{thm6.1}
Let $(A, \cdot_A)$ be a  Perm algebra and $P$ be a vector space equipped with a bilinear operation $\cdot_P : P \otimes P \to P$. 
Define the bracket on $A \otimes P$ by
\[
[a_1 \otimes p_1, a_2 \otimes p_2] := (a_1 \cdot_A a_2) \otimes (p_1 \cdot_P p_2) - (a_2 \cdot_A a_1) \otimes (p_2 \cdot_P p_1).
\]
Then:
\begin{enumerate}
    \item If $(P, \cdot_P)$ is a pre-Lie algebra, then $A \otimes P$ is a Lie algebra.
    \item Conversely, if $A \otimes P$ is a Lie algebra and $A$ is \textbf{free as a Perm algebra}, then $(P, \cdot_P)$ is a pre-Lie algebra.
\end{enumerate}
In particular, when $A$ is a free Perm algebra, $A \otimes P$ is a Lie algebra \textbf{if and only if} $(P, \cdot_P)$ is a pre-Lie algebra.
\end{theorem}

\begin{theorem}
Let $A$ be a free Perm algebra, $P$ be a pre-Lie  algebra, and $N$ be a vector space. Then the tensor product $A \otimes N$ with action
\[
(a \otimes n) \cdot (b \otimes c) := (a  b) \otimes (n  c) - (b  a)\otimes(c n )
\]
\[
(b \otimes c) \cdot (a \otimes n) := (b  a) \otimes (c  n) - (a  b)\otimes(n c)
\]
is a module over $A\otimes P$ if $N$ is a module over $P$. 
\end{theorem}
\begin{proof}
    This follows using an argument similar to the proof of Theorem  \ref{thm6.1}.
\end{proof}
\begin{definition}[\!\!\cite{ch}]
 Let \( \mathfrak{g} \) be a Lie algebra over a field \( k \), and let \( M \) be a \( \mathfrak{g} \)-module (i.e., a vector space equipped with a Lie algebra action of \( \mathfrak{g} \)). The Chevalley-Eilenberg complex is used to define the cohomology groups \( H^n(\mathfrak{g}, M) \).

The \( n \)-th cochain group \( C^n(\mathfrak{g}, M) \) consists of alternating \( n \)-linear maps (cochains):
\[
f: \mathfrak{g} \times \mathfrak{g} \times \cdots \times \mathfrak{g} \to M,
\]
and the coboundary operator 
\[
\delta_{Lie}: C^n(\mathfrak{g}, M) \to C^{n+1}(\mathfrak{g}, M)
\]
is given by:
\[
\begin{aligned}
    (\delta_{Lie}f)(x_1, \dots, x_{n+1}) &= \sum_{i \prec  j} (-1)^{i+j} f([x_i, x_j], x_1, \dots, \hat{x_i}, \dots, \hat{x_j}, \dots, x_{n+1}) \\
    & + \sum_{i} (-1)^{i+1} x_i \cdot f(x_1, \dots, \hat{x_i}, \dots, x_{n+1}),
\end{aligned}
\]
where \( \hat{x_i} \) means that \( x_i \) is omitted.

For $n=1$:
\[
(\delta_{Lie}f)(x_1, x_2) = -f([x_1, x_2]) + x_1 \cdot f(x_2) - x_2 \cdot f(x_1).
\]
For $n=2$:
\[
\begin{aligned}
    (\delta_{Lie}f)(x_1, x_2, x_3) &= -f([x_1, x_2], x_3) + f([x_1, x_3], x_2) - f([x_2, x_3], x_1) \\
    &+ x_1 \cdot f(x_2, x_3) - x_2 \cdot f(x_1, x_3) + x_3 \cdot f(x_1, x_2).
\end{aligned}
\]
\begin{example}

For the Lie algebra \(\mathfrak{g}\) with basis \(\{e_1, e_2\}\) and bracket \([e_1, e_2] = e_1\), all other brackets zero: \[H^2(\mathfrak{g},\mathfrak{g}) \cong \Bbbk.\]
\end{example}
\end{definition}
\begin{theorem}
Let \((A, \cdot_A)\) be a \textbf{Perm algebra}, let \((P, \prec, \succ)\) be a \textbf{pre-Lie algebra}, and let \(N\) be a \(P\)-bimodule. Then there exists a cochain map
\[
\Psi : C^*_{\mathrm{Pre}}(P, N) \rightarrow C^*_{\mathrm{Lie}}(A \otimes P, A \otimes N)
\]
from the pre-Lie algebra cochain complex of \(P\) with coefficients in \(N\) to the Lie cochain complex of the tensor product algebra \(A \otimes P\) with coefficients in \(A \otimes N\).
\end{theorem}
\begin{proof}
     Given a pre-Lie  cochain $f \in C^n_{\mathrm{Pre}}(P, N)$, for $x_1,\dots,x_n\in A$ and $b_1,\dots,b_n\in P$, define:
\begin{longtable}{@{}l@{}l@{}}
$\displaystyle \Psi(f)\big((x_1 \otimes b_1, x_2 \otimes b_2,$
& $\dots,x_{n}\otimes b_{n})\big) :=\displaystyle x_1x_2\dots x_{n} \otimes f(b_1 , b_2,\dots,b_{n})$ \\[4pt]
& $\displaystyle - x_2x_1x_3\dots x_{n} \otimes f(b_2, b_1,\dots,b_{n})$ \\[4pt]
& $\displaystyle + \sum_{i=2}^{n-2}(-1)^i x_{i+1}\dots x_{1}x_{i+2}\dots x_{n}\otimes f(b_1,\dots b_{i-1},b_{i+1},b_i,b_{i+2},\dots ,b_{n})$ \\[4pt]
& $\displaystyle + (-1)^{n-1}x_nx_{n-1}\dots x_2x_1\otimes f(b_n,b_{n-1},\dots ,b_2,b_1)$
\end{longtable}
 
We will show that $\Psi$ commutes with the  differentials:
\[
\delta_{Lie}(\Psi f)=\Psi(\delta_{Pre}f)
\]
\begin{itemize}
    \item 

For $n=2$
\[
\begin{aligned}
\delta_{\text{Lie}}(\Psi f)(x_1 \otimes b_1, x_2 \otimes b_2) &= -\Psi f([x_1 \otimes b_1, x_2 \otimes b_2]) \\
&+ (x_1 \otimes b_1) \cdot \Psi f(x_2 \otimes b_2) - (x_2 \otimes b_2) \cdot \Psi f(x_1 \otimes b_1)\\
&=- (x_1 \cdot_A x_2) \otimes  f\big((b_1 \cdot_P b_2)\big) + (x_2 \cdot_A x_1)\otimes  f\big((b_2 \cdot_P b_1)\big)\\
&+(x_1 \cdot_A x_2)\otimes b_1f(b_2)-(x_2 \cdot_A x_1)\otimes f(b_2)b_1\\
&-(x_2 \cdot_A x_1)\otimes b_2f(b_1)+(x_1 \cdot_A x_2)\otimes f(b_1)b_2\\
&=(x_1 \cdot_A x_2)\otimes\Big( -f(b_1 \cdot_P b_2)+b_1f(b_2)+f(b_1)b_2\Big)\\
&-(x_2 \cdot_A x_1)\Big(-f(b_2 \cdot_P b_1)+b_2f(b_1)+f(b_2)b_1\Big)\\
&=(x_1 \cdot_A x_2)\otimes (\delta_{pre}f)(b_1,b_2)-(x_2 \cdot_A x_1)\otimes (\delta_{pre}f)(b_2,b_1)\\
&=\Psi(\delta_{Pre}f)(x_1 \otimes b_1, x_2 \otimes b_2)
\end{aligned}
\]
\item For $n=3.$
 Apply \(\delta_{\text{Lie}}\) to \(\Psi(f)\):
   \[
   (\delta_{\text{Lie}} \Psi(f))(x_1 \otimes b_1, x_2 \otimes b_2, x_3 \otimes b_3) 
   \]
   has six terms as per the definition of \(\delta_{\text{Lie}}\).

1. First Term: \(- \Psi(f)([x_1 \otimes b_1, x_2 \otimes b_2], x_3 \otimes b_3)\):
   - Compute the bracket:
     \[
     [x_1 \otimes b_1, x_2 \otimes b_2] = (x_1 x_2) \otimes (b_1 \cdot_P b_2) - (x_2 x_1) \otimes (b_2 \cdot_P b_1).
     \]
   - Apply \(\Psi(f)\):
     \[
     \Psi(f)\big(([x_1 \otimes b_1, x_2 \otimes b_2], x_3 \otimes b_3)\big) 
     = \Psi(f)\big((x_1 x_2 \otimes (b_1 \cdot_P b_2) - x_2 x_1 \otimes (b_2 \cdot_P b_1), x_3 \otimes b_3)\big).
     \]
     This splits into two applications of \(\Psi(f)\):
     \[
     \begin{aligned}
     -\Psi(f)\big(([x_1 \otimes b_1, x_2 \otimes b_2], x_3 \otimes b_3)\big) &= -(x_1 x_2) x_3 \otimes f(b_1 \cdot_P b_2, b_3) \\
     &+ x_3 (x_1 x_2) \otimes f(b_3, b_1 \cdot_P b_2) \\
     &+ (x_2 x_1) x_3 \otimes f(b_2 \cdot_P b_1, b_3) \\
     &- x_3 (x_2 x_1) \otimes f(b_3, b_2 \cdot_P b_1).
      \end{aligned}
     \]

2. Second Term: 
     \[
     \begin{aligned}
     \Psi(f)\big(([x_1 \otimes b_1, x_3 \otimes b_3], x_2 \otimes b_2)\big) 
     &= (x_1 x_3) x_2 \otimes f(b_1 \cdot_P b_3, b_2) \\
     &- x_2 (x_1 x_3) \otimes f(b_2, b_1 \cdot_P b_3)\\ 
     &+ (x_3 x_1) x_2 \otimes f(b_3 \cdot_P b_1, b_2)\\
     &- x_2 (x_3 x_1) \otimes f(b_2, b_3 \cdot_P b_1).
      \end{aligned}
     \]

3. Third Term:
     \[
      \begin{aligned}
    - \Psi(f)\big(([x_2 \otimes b_2, x_3 \otimes b_3], x_1 \otimes b_1)\big) 
     &= -(x_2 x_3) x_1 \otimes f(b_2 \cdot_P b_3, b_1) \\
     & + x_1 (x_2 x_3) \otimes f(b_1, b_2 \cdot_P b_3) \\
     &+ (x_3 x_2) x_1 \otimes f(b_3 \cdot_P b_2, b_1) \\
     & - x_1 (x_3 x_2) \otimes f(b_1, b_3 \cdot_P b_2).
     \end{aligned}
     \]

4. Fourth Term: \(+ (x_1 \otimes b_1) \cdot \Psi(f)(x_2 \otimes b_2, x_3 \otimes b_3)\):

   - First compute \(\Psi(f)(x_2 \otimes b_2, x_3 \otimes b_3)\):
     \[
     \Psi(f)(x_2 \otimes b_2, x_3 \otimes b_3) = x_2 x_3 \otimes f(b_2, b_3) - x_3 x_2 \otimes f(b_3, b_2).
     \]
   - Then act by \(x_1 \otimes b_1\):
     \[
     \begin{aligned}
     (x_1 \otimes b_1) \cdot (x_2 x_3 \otimes f(b_2, b_3) - x_3 x_2 \otimes f(b_3, b_2)) 
     &= x_1 (x_2 x_3) \otimes b_1  f(b_2, b_3) \\
     &-  (x_2 x_3) x_1\otimes  f(b_2, b_3)b_1.\\
     &- x_1 (x_3 x_2) \otimes b_1  f(b_3, b_2) \\
     &+  (x_3 x_2) x_1\otimes  f(b_3, b_2)b_1.
      \end{aligned}
     \]
5. Fifth Term:
     \[
     \begin{aligned}
     -(x_2 \otimes b_2) \cdot (x_1 x_3 \otimes f(b_1, b_3) - x_3 x_1 \otimes f(b_3, b_1)) 
     &= -x_2 (x_1 x_3) \otimes b_2  f(b_1, b_3)\\
     & +(x_1x_3)x_2 \otimes   f(b_1, b_3)b_2\\
     &+ x_2 (x_3 x_1) \otimes b_2  f(b_3, b_1)\\
     & - (x_3 x_1)x_2 \otimes   f(b_3, b_1)b_2.
      \end{aligned}
     \]

6. Sixth Term:
     \[
     \begin{aligned}
     (x_3 \otimes b_3) \cdot (x_1 x_2 \otimes f(b_1, b_2) - x_2 x_1 \otimes f(b_2, b_1)) 
     &= x_3 (x_1 x_2) \otimes b_3  f(b_1, b_2)\\
     & - (x_1 x_2)x_3 \otimes f(b_1, b_2)b_3\\
     &- x_3 (x_2 x_1) \otimes b_3  f(b_2, b_1)\\
     &+ (x_2 x_1)x_3 \otimes  f(b_2, b_1)b_3.
     \end{aligned}
     \]

 Final Expression: $ (\delta_{\text{Lie}} \Psi(f))(x_1 \otimes b_1, x_2 \otimes b_2, x_3 \otimes b_3)$
    \begin{longtable}{l}
        $\displaystyle = x_1x_2x_3\otimes\Big(-b_1 \cdot f(b_2, b_3) +b_1 \cdot f(b_3, b_2) +f(b_1 \cdot_P b_2, b_3) - f(b_1 \cdot_P b_3, b_2)$ \\
        $\displaystyle \phantom{= x_1x_2x_3\otimes\Big(} + f(b_1, [b_2, b_3]) - f(b_1, b_2) \cdot b_3 + f(b_1, b_3) \cdot b_2\Big)$ \\[4pt]
        $\displaystyle -x_2x_1x_3 \otimes\Big(-b_2 \cdot f(b_1, b_3) +b_2 \cdot f(b_3, b_1) +f(b_2 \cdot_P b_1, b_3) - f(b_2 \cdot_P b_3, b_1)$ \\
        $\displaystyle \phantom{-x_2x_1x_3 \otimes\Big(} + f(b_2, [b_1, b_3]) - f(b_2, b_1) \cdot b_3 + f(b_2, b_3) \cdot b_1\Big)$ \\[4pt]
        $\displaystyle + x_3x_2x_1\otimes\Big(-b_3 \cdot f(b_2, b_1) +b_3 \cdot f(b_1, b_2) +f(b_3 \cdot_P b_2, b_1) - f(b_3 \cdot_P b_1, b_2)$ \\
        $\displaystyle \phantom{+ x_3x_2x_1\otimes\Big(} + f(b_3, [b_2, b_1]) - f(b_3, b_2) \cdot b_1 + f(b_3, b_1) \cdot b_2\Big)$ \\[4pt]
        $\displaystyle =x_1x_2x_3\otimes (\delta_{Pre}f)(b_1,b_2,b_3)-x_2x_1x_3\otimes (\delta_{Pre}f)(b_2,b_1,b_3)$ \\
        $\displaystyle \phantom{=}+x_3x_2x_1\otimes (\delta_{Pre}f)(b_3,b_2,b_1)=\Psi(\delta_{Pre}f) (x_1 \otimes b_1, x_2 \otimes b_2, x_3 \otimes b_3)$
    \end{longtable}
        \item General case: Following the same pattern of proof used for the cases $n=2$ and $n=3$, the theorem can be established for all $n \geq 2$.
        \qedhere
        \end{itemize}
        \end{proof}
        \begin{remark}
        
 In the case where the Perm algebra \(A\) is free (see Remark \ref{re:3.5}), the cochain map \(\Psi\) is injective.
\end{remark}

\begin{corollary}
    Given the canonical embedding of cochain complexes
\[
C^*_{\mathrm{Pre}}(P,N) \hookrightarrow C^*_{\mathrm{Lie}}(A\otimes P,A\otimes N),
\]
we obtain a short exact sequence of complexes:
\[
0 \to C^*_{\mathrm{Pre}}(P,N) \hookrightarrow C^*_{\mathrm{Lie}}(A\otimes P,A\otimes N)\to Q^* \to 0,
\]
where $ Q^* = C^*_{\mathrm{Lie}}(A\otimes P,A\otimes N)/C^*_{\mathrm{Pre}}(P,N)$ is the quotient complex.
Applying the cohomology functor yields the long exact sequence:
\[
\begin{aligned}
\cdots \to H^{n-1}_{\mathrm{Pre}}(P,N) &\to H^{n-1}_{Lie}(A\otimes P,A\otimes N) \to H^{n-1}(Q^*) \\
&\to H^n_{\mathrm{Pre}}(P,N) \to H^n_{Lie}(A\otimes P,A\otimes N) \to H^n(Q^*) \\
&\to H^{n+1}_{Pre}(P,N) \to H^{n+1}_{Lie}(A\otimes P,A\otimes N) \to \cdots
\end{aligned}
\]
\end{corollary}
 \begin{example}
Let:
\begin{itemize}
\item $A = F\langle x,y\rangle$ be the free Perm algebra.
\item $P$ be  pre-Lie algebra with basis $\{e_1,e_2\}$ and operations:
\begin{align*}
e_1\cdot_P e_2 &= e_1 \\
e_i \cdot_P e_j &= 0\quad \textit{for other.}
\end{align*}
\end{itemize}
A pre-Lie 2-cochain \( f:  P \otimes  P \to  P \) can be expressed as (see  \ref{exm:5.4}):
\[
f(e_1, e_2) = a e_1, \quad f(e_2, e_1) = c e_1.
\] 
A Lie 2-cochain \( \phi: (A \otimes P) \otimes (A \otimes P) \to (A \otimes P) \) can be expressed in terms of \( f \) as:
\[
\phi(x \otimes e_1, y \otimes e_2) = a xy \otimes e_1 + c yx \otimes e_1.
\] 
\end{example}


{\small
    \begin{thebibliography}{1}
        \bibitem{Gnedbaye2020}
        V.~G. Allahtan, K.~Aslaou, and D.~Djagwa.
        \newblock Differential forms and cohomology of {Perm} algebras.
        \newblock {\em IOSR Journal of Mathematics (IOSR-JM)}, 16(1):9--20, 2020.
        
        \bibitem{Chapoton2001}
        F.~Chapoton.
        \newblock Un endofoncteur de la cat\'egorie des op\'erades.
        \newblock In {\em Dialgebras and Related Operads}, volume 1763 of {\em Lecture
          Notes in Mathematics}, pages 105–110. Springer-Verlag, 2001.
        
        \bibitem{ch}
        C.~Chevalley and S.~Eilenberg.
        \newblock Cohomology theory of Lie groups and Lie algebras.
        \newblock {\em Transactions of the American Mathematical Society},
          63(1):85--124, 1948.
        
        \bibitem{Das}
        A.~Das.
        \newblock Cohomology and deformations of dendriform algebras, and ${\rm Dend}_\infty$-algebras.
        \newblock {\em Communications in Algebra}, 50(4):1544--1567, 2022.
        
        \bibitem{Dzhumadil’daev}
        A.~Dzhumadil'daev.
        \newblock Cohomologies and deformations of right-symmetric algebras.
        \newblock {\em Journal of Mathematical Sciences}, 93(6):836--876, 1999.
        
        \bibitem{Gerstenhaber}
        M.~Gerstenhaber.
        \newblock The cohomology structure of an associative ring.
        \newblock {\em Annals of Mathematics}, 78(2):267--288, 1963.
        
        \bibitem{GubarevKolesnikov}
        V.~Y. Gubarev and P.~S. Kolesnikov.
        \newblock Operads of decorated trees and their duals.
        \newblock {\em Commentarii Mathematici Universitatis Carolinae},
          55(4):421--445, 2014.
        
        \bibitem{kolesnikov}
        P.~S. Kolesnikov.
        \newblock Gr\"obner-Shirshov bases for pre-associative algebras.
        \newblock {\em Communications in Algebra}, 45(12):5283--5296, 2017.
        
        \bibitem{Loday}
        J.-L. Loday.
        \newblock Dialgebras.
         \newblock In {\em Dialgebras and Related Operads}, volume 1763 of {\em Lecture
          Notes in Mathematics}, pages 7--66. Springer-Verlag, 2001.
    \end{thebibliography}
}
\end{document}